\documentclass{article}
\usepackage[utf8]{inputenc}
\usepackage[margin=1 in]{geometry}
\usepackage{amsmath}
\usepackage{amsfonts}
\usepackage{amssymb}
\usepackage{amsthm}
\usepackage{mathrsfs}
\usepackage{enumitem}
\usepackage{hyperref}
\usepackage[vcentermath,enableskew]{youngtab}
\usepackage{ytableau}

\newtheorem{theorem}{Theorem}
\newtheorem{lemma}[theorem]{Lemma}

\theoremstyle{definition}

\newtheorem{remark}[theorem]{Remark}
\newtheorem{example}[theorem]{Example}

\newcommand{\tb}{\textbf}

\newcommand{\Z}{\mathbb{Z}}

\newcommand{\tn}{\textnormal}
\newcommand{\se}{\subseteq}

\newcommand{\ex}{\tn{ex}}

\newcommand{\ol}{\overline}
\newcommand{\lam}{\lambda}

\newcommand{\on}{\operatorname}
\newcommand{\ap}{\ol{\mf{A}}2\ol{\mf{P}}}

\newcommand{\mf}{\mathfrak}
\newcommand{\tcb}{\textcolor{blue}}

\usepackage[maxbibnames=10,style=alphabetic]{biblatex}
\title{Proof of a $K$-theoretic polynomial conjecture of Monical, Pechenik, and Searles}
\author{Laura Pierson \\ University of Waterloo \\ \href{mailto:lcpierson73@gmail.com}{lcpierson73@gmail.com}}
\begin{document}

\maketitle

\begin{abstract}
    As part of a program to develop $K$-theoretic analogues of combinatorially important polynomials, Monical, Pechenik, and Searles (2021) proved two expansion formulas $\ol{\mf{A}}_a = \sum_b Q_b^a(\beta)\ol{\mf{P}}_b$ and $\ol{\mf{Q}}_a = \sum_b M_b^a(\beta)\ol{\mf{F}}_b,$ where each of $\ol{\mf{A}}_a$, $\ol{\mf{P}}_a$, $\ol{\mf{Q}}_a$ and $\ol{\mf{F}}_a$ is a family of polynomials that forms a basis for $\Z[x_1,\dots,x_n][\beta]$ indexed by weak compositions $a,$ and $Q_b^a(\beta)$ and $M_b^a(\beta)$ are monomials in $\beta$ for each pair $(a,b)$ of weak compositions. The polynomials $\ol{\mf{A}}_a$ are the \emph{\tb{\tcb{Lascoux atoms}}}, $\ol{\mf{P}}_a$ are the \emph{\tb{\tcb{kaons}}}, $\ol{\mf{Q}}_a$ are the \emph{\tb{\tcb{quasiLascoux polynomials}}}, and $\ol{\mf{F}}_a$ are the \emph{\tb{\tcb{glide polynomials}}}; these are respectively the $K$-analogues of the \emph{\tb{\tcb{Demazure atoms}}} $\mf{A}_a$, the \emph{\tb{\tcb{fundamental particles}}} $\mf{P}_a$, the \emph{\tb{\tcb{quasikey polynomials}}} $\mf{Q}_a$, and the \emph{\tb{\tcb{fundamental slide polynomials}}} $\mf{F}_a$. Monical, Pechenik, and Searles conjectured that for any fixed $a,$ $\sum_b Q_b^a(-1), \sum_b M_b^a(-1) \in \{0,1\},$ where $b$ ranges over all weak compositions. We prove this conjecture using a sign-reversing involution.
\end{abstract}

\section{Introduction}

We prove a conjecture of Monical, Pechenik, and Searles from \cite{Monical_Pechenik_Searles_2021} stating that two alternating sums involving polynomials arising from \emph{\tb{\tcb{combinatorial $K$-theory}}} are always equal to either 0 or 1. These polynomials were defined as part of a project to extend the combinatorics associated to the ring ${\sf{Sym}}_n$ of \emph{\tb{\tcb{symmetric polynomials}}} in $n$ variables to the larger ring ${\sf{Poly}}_n := \Z[x_1,\dots,x_n]$ of arbitrary polynomials in $n$ variables, and then further to the ring ${\sf{Poly}}_n[\beta]$ of $K$-theoretic deformations of these polynomials. The motivation for defining these $K$-theoretic analogues arises from the connections these polynomials have to algebraic geometry, specifically to \emph{\tb{\tcb{Schubert calculus}}}, which studies the cohomology rings of \emph{\tb{\tcb{flag varieties}}}. The basic idea of $K$-theory is to deform the structure of cohomology rings by introducing an additional parameter $\beta$. Polynomials coming from cohomology rings often also have corresponding $K$-analogues, and conversely, when polynomials arising from combinatorics have nice $K$-analogues, that indicates that they may have geometric interpretations. See \cite{buch2005combinatorial} for an introduction to combinatorial $K$-theory.

A classic basis for ${\sf{Sym}}_n$ is the \emph{\tb{\tcb{Schur polynomials}}} $s_\lam$, which have rich connections to representation theory and algebraic geometry. These polynomials are indexed by \emph{\tb{\tcb{partitions}}} $\lam=(\lam_1,\lam_2,\dots,\lam_n)$, i.e. finite nondecreasing sequences of positive integers $\lam_1 \ge \lam_2 \ge \dots \ge \lam_n$. Schur polynomials are a special case of \emph{\tb{\tcb{Schubert polynomials}}} $\mf{S}_a$, which were introduced by Lascoux and Sch\"utzenberger in \cite{LS82} as a basis for ${\sf{Poly}}_n$ that is closely connected to Schubert calculus. The Schubert polynomials are indexed by \emph{\tb{\tcb{weak compositions}}} $a=(a_1,\dots,a_n),$ i.e. finite sequences of nonnegative integers. Various related bases for ${\sf{Poly}}_n$ have also been studied, each also indexed by weak compositions $a$, including the \emph{\tb{\tcb{Demazure characters}}} $\mf{D}_a$ \cite{Dem74}, the \emph{\tb{\tcb{Demazure atoms}}} $\mf{A}_a$ \cite{LS90}, the \emph{\tb{\tcb{(fundamental) slide polynomials}}} $\mf{F}_a$ \cite{AS17}, the \emph{\tb{\tcb{(fundamental) particles}}} $\mf{P}_a$ \cite{Sea17}, and the \emph{\tb{\tcb{quasikey polynomials}}} $\mf{Q}_a$ \cite{AS18b}. Hicks and Niese showed in \cite{hicks2024quasisymmetric} that all these polynomials arise naturally using a quasisymmetric analogue of the \emph{\tb{\tcb{divided difference operator}}}, and some of these polynomials also arise naturally from representation theory. These bases all have $K$-theoretic lifts to corresponding bases for ${\sf{Poly}}_n[\beta]$: the \emph{\tb{\tcb{Grothendieck polynomials}}} $\ol{s}_\lam$ and $\ol{\mf{S}}_a$ \cite{FK94}, the \emph{\tb{\tcb{Lascoux polynomials}}} $\ol{\mf{D}}_a$, the \emph{\tb{\tcb{Lascoux atoms}}} $\ol{\mf{A}}_a$ \cite{Mon16}, the \emph{\tb{\tcb{glide polynomials}}} $\ol{\mf{F}}_a$ \cite{PS17}, the \emph{\tb{\tcb{kaons}}} $\ol{\mf{P}}_a$ \cite{Monical_Pechenik_Searles_2021}, and the \emph{\tb{\tcb{quasiLascoux polynomials}}} $\ol{\mf{Q}}_a$ \cite{Monical_Pechenik_Searles_2021}. An integrable vertex model for Lascoux polynomials and Lascoux atoms has been developed in \cite{Buciumas.Scrimshaw.Weber}, using ideas from statistical mechancis.

The authors of \cite{Monical_Pechenik_Searles_2021} proved two expansion formulas involving these $K$-theoretic polynomials: $$\ol{\mf{A}}_a = \sum_b Q_b^a(\beta) \ol{\mf{P}}_b, \hspace{2cm} \ol{\mf{Q}}_a = \sum_b M_b^a(\beta) \ol{\mf{F}}_b,$$ where each $Q_b^a(\beta)$ and $M_b^a(\beta)$ is a monomial of the form $m \beta^n$ for nonnegative integers $m$ and $n$ (\cite{Monical_Pechenik_Searles_2021}, Theorems 3.12 and 4.12). They then conjectured that for all $b$, $\sum_b Q_b^a(-1)\in\{0,1\}$ and $\sum_b M_b^a(-1)\in\{0,1\}.$ We prove this conjecture and also characterize exactly when the sum is 0 and when it is 1:

\begin{theorem}[cf. \cite{Monical_Pechenik_Searles_2021}, Conjecture 1.4]\label{thm:alt_sum}
    If the nonzero parts of $a$ are weakly decreasing, then $$\sum_b Q_b^a(-1) = \sum_b M_b^a(-1) = 1,$$ and otherwise $$\sum_b Q_b^a(-1) = \sum_b M_b^a(-1) = 0,$$ where the sums are taken over all weak compositions $b.$
\end{theorem}
Theorem~\ref{thm:alt_sum} is suggestive of an Euler characteristic calculation and may provide some insight into the geometric context of these polynomials.

We will explain the relevant definitions and prior results in \S \ref{sec:alt_sum_background}, and then in \S \ref{sec:alt_sum_proof}, we will prove Theorem \ref{thm:alt_sum} and give some examples to illustrate it.

\section{Background}\label{sec:alt_sum_background}

In this section, we will define the four families of polynomials studied here, in the following order:
\begin{enumerate}
    \item The \emph{\tb{\tcb{kaons}}} $\ol{\mf{P}}_a$, which are generating functions for \emph{\tb{\tcb{mesonic glides}}}.
    \item The \emph{\tb{\tcb{Lascoux atoms}}} $\ol{\mf{A}}_a$, which are generating functions for \emph{\tb{\tcb{set-valued skyline fillings}}}.
    \item The \emph{\tb{\tcb{glide polynomials}}} $\ol{\mf{F}}_a$, which are sums of kaons.
    \item The \emph{\tb{\tcb{quasiLascoux polynomials}}} $\ol{\mf{Q}}_a,$ which are sums of Lascoux atoms.
\end{enumerate}
We will also explain the expansion formulas from \cite{Monical_Pechenik_Searles_2021} that express Lascoux atoms in terms of kaons and quasiLascoux atoms in terms of glide polynomials, which will be the main tools for our proof.

\begin{remark}
    The overlines indicate that these are $K$-theoretic polynomials (i.e. they involve $\beta$). The corresponding versions without the overlines are obtained by setting $\beta = 0.$ The letter $\mf{A}$ is an ``A" for ``atom." The letter $\mf{Q}$ is a ``Q" for ``quasi." The letter $\mf{P}$ is a ``P" for ``particle," because the kaons $\ol{\mf{P}}_a$ are a $K$-analogue of the particles $\mf{P}_a.$ The letter $\mf{F}$ is an ``F" for ``fundamental," because the glides $\ol{\mf{F}}_a$ are a $K$-analogue of the fundamental slides $\mf{F}_a,$ which lift the fundamental basis $F_a$ of the ring ${\sf{QSym}}$ of quasisymmetric functions. The name ``atoms" makes sense because the Lascoux atoms form a basis of ${\sf{Poly}}_n[\beta]$ and hence a set of fundamental building blocks. ``Kaons" are named after a type of subatomic particle, which fits because they are ``subatomic" in the sense that Lascoux atoms can be built out of kaons, and also because in particle physics, the symbol for kaons is ``K." The name ``meson" comes from the name in particle physics of the family of particles that includes kaons (kaons are also called ``K mesons"). The name ``skyline fillings" comes from the fact that the diagrams look like city skylines  turned sideways.
\end{remark}



\subsection{Kaons \texorpdfstring{$\ol{\mf{P}}_a$}{P}}

First we introduce the definitions needed for the kaons. A \emph{\tb{\tcb{weak komposition}}} $b=(b_1,\dots,b_n)$ is a weak composition whose entries are colored red or black, and its \emph{\tb{\tcb{excess}}} $\ex(b)$ is the number of red entries. Let $a=(a_1,\dots,a_n)$ be a weak composition with nonzero entries at positions $n_1 < n_2 < \dots < n_\ell$. We also set $n_0 = 0$. A weak komposition $b$ is a \emph{\tb{\tcb{mesonic glide}}} of $a$ if the following three properties are satisfied for every $1\le j \le \ell$:
\begin{enumerate}[leftmargin=1.5cm]
    \item[(G.$1'$)] $a_{n_j} = b_{n_{j-1}+1}+\dots + b_{n_j} - \tn{ex}(b_{n_{j-1}+1},\dots, b_{n_j})$.
    \item[(G.$3'$)] The leftmost nonzero entry of $(b_{n_{j-1}+1},\dots, b_{n_j})$ is black.
    \item [(G.$4'$)] $b_{n_j}\ne 0.$
\end{enumerate}

In our examples below, the red entries are shown in bold with a line over them to help them stand out.

\begin{example}
    The weak komposition $b = (2,{\bf{\color{red}\ol{2}}},1,{\bf{\color{red}\ol{3}}},{{\bf{\color{red}\ol{2}}}})$ is a mesonic glide of $a = (0,3,0,0,4),$ since 
    \begin{align*}
        a_{n_1} &= a_2 = 3 = 2 + {\bf{\color{red}\ol{2}}} - \tn{ex}(2,{\bf{\color{red}\ol{2}}}), \\
        a_{n_2} &= a_5 = 4 = 1 + {\bf{\color{red}\ol{3}}} + {\bf{\color{red}\ol{2}}} - \ex(1,{\bf{\color{red}\ol{3}}},{\bf{\color{red}\ol{2}}}),
    \end{align*}
    the leftmost entries of $(2,{\bf{\color{red}\ol{2}}})$ and $(1,{\bf{\color{red}\ol{3}}},{\bf{\color{red}\ol{2}}})$ are black, $b_{n_1} = b_2 = {\bf{\color{red}\ol{2}}}\ne 0,$ and $b_{n_2} = b_5 = {\bf{\color{red}\ol{2}}}\ne 0.$
\end{example}

Write $\mathbf{x}^{b} := x_1^{b_1}\dots x_n^{b_n}.$ The \emph{\tb{\tcb{kaon}}} $\ol{\mf{P}}_a$ is the generating function $$\ol{\mf{P}}_a := \sum_{b\tn{ a mesonic glide of }a} \beta^{\ex(b)}\mathbf{x}^b.$$ In the example below, we show the $x$ variables corresponding to bolded red entries as also being bolded and red to help illustrate the correspondence. This coloring of the $x$ variables is not actually part of the definition of the polynomials.

\begin{example}
    For $a=(0,2,0,1),$ the possible mesonic slides are $$\begin{array}{cccc}
        (0,2,0,1) & (1,1,0,1) & (1,{\bf{\color{red}\ol{2}}},0,1) & (2,{\bf{\color{red}\ol{1}}},0,1) \\
        (0,2,1,{\bf{\color{red}\ol{1}}}) & (1,1,1,{\bf{\color{red}\ol{1}}}) & (1,{\bf{\color{red}\ol{2}}},1,{\bf{\color{red}\ol{1}}}) & (2,{\bf{\color{red}\ol{1}}},1,{\bf{\color{red}\ol{1}}}) \\
    \end{array}$$
    so the corresponding kaon is $$\ol{\mf{P}}_{(0,2,0,1)} = x_2^2x_4 + x_1x_2x_4 + \beta x_1{\color{red}{\boldsymbol{\ol{x_2}^2}}} x_4 + \beta x_1^2{\color{red}{\boldsymbol{\ol{x_2}}}}x_4 + \beta x_2^2x_3{\color{red}{\boldsymbol{\ol{x_4}}}} + \beta x_1x_2x_3{\color{red}{\boldsymbol{\ol{x_4}}}} + \beta^2 x_1{\color{red}{\boldsymbol{\ol{x_2}^2}}} x_3 {\color{red}{\boldsymbol{\ol{x_4}}}} + \beta^2 x_1^2 {\color{red}{\boldsymbol{\ol{x_2}}}}x_3 {\color{red}{\boldsymbol{\ol{x_4}}}}.$$
\end{example}

\begin{remark}
    We follow the seemingly odd numbering (G.$1'$), (G.$3'$), and (G.$4'$) to match the notation in \cite{Monical_Pechenik_Searles_2021}. They use this notation because the conditions (G.$1'$), (G.$3'$), and (G.$4'$) are a modification of the conditions (G.1), (G.2), and (G.3) used to define the glide polynomials $\ol{\mf{F}}_a,$ which were originally defined before the kaons $\ol{\mf{P}}_a$. We choose here to define the kaons $\ol{\mf{P}}_a$ first and then the glides $\ol{\mf{F}}_a$, since it is simpler to define the glides $\ol{\mf{F}}_a$ in terms of the kaons $\ol{\mf{P}}_a$. 
\end{remark}

\begin{remark}
    Both sets of glide conditions (G.1), (G.2), (G.3) and (G.$1'$), (G.$3'$), (G.$4'$) are themselves modifications of a simpler set of \emph{\tb{\tcb{slide conditions}}}, which are used to define the fundamental particles $\mf{P}_a$ and the fundamental slide polynomials $\mf{F}_a$. The particles $\mf{P}_a$ and slides $\mf{F}_a$ are the respective $\beta = 0$ cases of $\ol{\mf{P}}_a$ and $\ol{\mf{F}}_a$ and hence do not involve the red entries. These sorts of slide and glide moves that allow entries to be moved in only one direction arise naturally from the study of the ring of \emph{\tb{\tcb{quasisymmetric polynomials}}} ${\sf QSym}_n$, which sits between ${\sf Sym}_n$ and ${\sf Poly}_n$ (i.e. ${\sf Sym}_n \subset {\sf QSym}_n \subset {\sf Poly}_n$). Much of the combinatorics of ${\sf Sym}_n$ was originally extended to ${\sf QSym}_n$ before it was extended to ${\sf Poly}_n$. Adding the red entries to the ``slides" to create the ``glides" turns out to be the natural way to introduce the extra parameter $\beta.$
\end{remark}

\subsection{Lascoux atoms \texorpdfstring{$\ol{\mf{A}}_a$}{A}}

Next, we introduce the definitions needed for the Lascoux atoms. A \emph{\tb{\tcb{skyline diagram}}} of shape $a = (a_1,\dots, a_n)$ consists of $n$ left justified rows of boxes such that the $i$th row from the bottom has $a_i$ boxes. A \emph{\tb{\tcb{set-valued skyline filling}}} $T$ of shape $a$ is a filling of each box of the skyline diagram of $a$ with a nonempty set of positive integers. The largest entry of each box is called the \emph{\tb{\tcb{anchor}}}, and the other entries are \emph{\tb{\tcb{free}}}. We say $T$ is \emph{\tb{\tcb{semistandard}}} if it satisfies the following properties:
\begin{itemize}[leftmargin=1.5cm]
    \item[(S.1)] There are no repeat entries in any column.
    \item[(S.2)] The entries weakly decrease across rows, in the sense that the smallest entry in each box is at least the largest entry in the box to its right.
    \item[(S.3)] Suppose the anchors $\alpha,\beta,$ and $\gamma$ satisfy all of the following conditions:
    \begin{itemize}
        \item $\alpha$ is immediately to the right of $\beta$.
        \item The row containing $\alpha$ and $\beta$ is at least as long as the row containing $\gamma$ if it is lower than the row containing $\gamma$, or is strictly longer if it is higher than the row containing $\gamma.$
        \item $\gamma$ is either above $\alpha$ in the same column or below $\beta$ in the same column.
    \end{itemize}
    Then we must have either $\gamma < \alpha$ or $\gamma > \beta$.

    Equivalently, this says that if $\alpha,\beta,$ and $\gamma$ are arranged in either an ``upside down L" or ``backwards L" as shown below such that the row containing the base of the L is longer, then $\gamma$ cannot be equal to or between $\alpha$ and $\beta$:
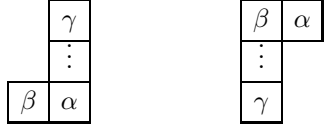
    $$\begin{ytableau}\none & \gamma \\ \none & \vdots \\ \beta & \alpha \end{ytableau} \hspace{2cm} \begin{ytableau} \beta & \alpha \\ \vdots & \none \\ \gamma & \none \end{ytableau}$$
     (Note that since $\alpha$ is to the right of $\beta,$ we must have $\alpha \le \beta$ by (S.2)).
    \item[(S.4)] Each free entry is assigned to the box in its column with the smallest possible anchor such that (S.2) is satisfied (so choosing the column of a free entry uniquely determines its row as well).
    \item[(S.5)] Each anchor in the leftmost column equals the index of its row.
\end{itemize}
We write $\ol{\mf{A}}{\sf{SSF}}(a)$ for the set of all semistandard skyline tableaux of shape $a.$ The \emph{\tb{\tcb{weight}}} $\on{wt}(T)$ of a tableau $T$ is the weak composition whose $i$th part is the number of times $i$ appears as an entry in $T$. 

The \emph{\tb{\tcb{Lascoux atom}}} $\ol{\mf{A}}_a$ is the generating function $$\ol{\mf{A}}_a := \sum_{T\in \ol{\mf{A}}{\sf{SSF}}(a)} \beta^{\ex(T)}\mathbf{x}^{\on{wt}(T)}.$$ 

\begin{remark}
    The Lascoux atoms are the $K$-analogue of the Demazure atoms $\mf{A}_a$, which are the $\beta=0$ case and hence involve only the anchors and no free entries. The Demazure atoms are a modification of the \emph{\tb{\tcb{Demazure characters}}} (also called \emph{\tb{\tcb{key polynomials}}}) from \cite{Dem74}, which were originally defined algebraically in terms of divided difference operators and were later shown in \cite{LS90} to also have a combinatorial description in terms of \emph{\tb{\tcb{key tableaux}}}. We can think of these key tableaux and skyline tableaux as nonsymmetric analogues of SSYT. Although the rules for defining these tableaux are more complicated, they turn out to have some nice associated combinatorics that is analogous to the combinatorics of SSYT, such as analogues of the Littlewood-Richardson rule for multiplying Schur polynomials (\cite{brubaker2021colored}).
\end{remark}



\subsection{Transition formula from Lascoux atoms \texorpdfstring{$\ol{\mf{A}}_a$}{A} to kaons \texorpdfstring{$\ol{\mf{P}}_a$}{P}}

We now give the definitions needed to state the $\ol{\mf{A}}_a$ to $\ol{\mf{P}}_a$ transition formula. A skyline tableau $T$ has the \emph{\tb{\tcb{meson-highest}}} property if for every $i$ that shows up as an entry in $T$, at least one of the following holds:
\begin{itemize}
    \item There is an $i$ that is an anchor in row $i$ (i.e. row $i$ is nonempty).
    \item There is an occurrence of the next smallest entry (which we denote $i^\uparrow$) that is weakly to the right of and in a different box from the leftmost $i.$
\end{itemize}
The set $\ap(a)$ consists of all semistandard meson-highest tableaux of shape $a.$ We write $|a|$ for the sum of the entries of $a$ and $|T|$ for the total number of entries of $T$. We can now state the first expansion formula:

\begin{theorem}[Monical, Pechenik, and Searles \cite{Monical_Pechenik_Searles_2021}, Theorem 3.12]\label{thm:a2p}
    The kaon expansion of the Lascoux atoms is $$\ol{\mf{A}}_a = \sum_{T\in \ap(a)}\beta^{|T|-|a|}\ol{\mf{P}}_{\tn{wt}(T)}.$$
\end{theorem}

\begin{example}\label{ex:a2p_0}
    The tableaux in $\ap(0,0,1,2)$ are below, grouped into columns by number $|T|-|a|$ of free entries (i.e. the exponent on $\beta$), with the anchors shown bolded and the other entries unbolded. The corresponding term in the kaon expansion of $\ol{\mf{A}}_{(0,0,2,2)}$ is listed below each tableau. Note that there are actually an extra two rows numbered 1 and 2 at the bottom of each of these tableaux, which we omit because they are empty.
    $$\arraycolsep=12pt\def\arraystretch{1.5}\begin{array}{llllll}
        \ytableausetup{boxsize=2em}
        \begin{ytableau}
            \tb{4} & \tb{4} \\
            \tb{3}
        \end{ytableau} & 
        \begin{ytableau}
            \tb{4} & \tb{4}3 \\
            \tb{3}
        \end{ytableau} &
        \begin{ytableau}
            \tb{4} & \tb{4}3 \\
            \tb{3}2
        \end{ytableau} &
        \begin{ytableau}
            \tb{4} & \tb{4}32 \\
            \tb{3}2
        \end{ytableau} &
        \begin{ytableau}
            \tb{4} & \tb{4}32 \\
            \tb{3}21
        \end{ytableau} &
        \begin{ytableau}
            \tb{4} & \tb{4}321 \\
            \tb{3}21
        \end{ytableau}\\
        \ol{\mf{P}}_{(0,0,1,2)} & \beta\cdot \ol{\mf{P}}_{(0,0,2,2)} & \beta^2\cdot\ol{\mf{P}}_{(0,1,2,2)} & \beta^3\cdot\ol{\mf{P}}_{(0,2,2,2)} & \beta^4\cdot \ol{\mf{P}}_{(1,2,2,2)} & \beta^5\cdot\ol{\mf{P}}_{(2,2,2,2)} \\ \\
        & & \begin{ytableau}
            \tb{4} & \tb{4}3 \\
            \tb{3}1
        \end{ytableau} & 
        \begin{ytableau}
            \tb{4} & \tb{4}31 \\
            \tb{3}1
        \end{ytableau} \\
        & & \beta^2\cdot\ol{\mf{P}}_{(1,0,2,2)} & \beta^3 \cdot \ol{\mf{P}}_{(2,0,2,2)} \\        
    \end{array}$$   
\end{example}

\begin{example}\label{ex:a2p_1}
    The tableaux in $\ol{\mf{A}}2\ol{\mf{P}}(0,0,2,2)$ are shown below, again grouped in columns by the power of $\beta,$ and with the corresponding monomial listed below each tableau. As in Example \ref{ex:a2p_0}, each tableau actually has an extra two rows numbered 1 and 2 at the bottom that are not drawn here.
    $$\arraycolsep=12pt\def\arraystretch{1.5}\begin{array}{llllll}
        \ytableausetup{boxsize=2em}
        \begin{ytableau}
            \tb{4} & \tb{4} \\
            \tb{3} & \tb{3}
        \end{ytableau} & 
        \begin{ytableau}
            \tb{4} & \tb{4}3 \\
            \tb{3} & \tb{2}
        \end{ytableau} &
        \begin{ytableau}
            \tb{4} & \tb{4}3 \\
            \tb{3}2 & \tb{2}
        \end{ytableau} &
        \begin{ytableau}
            \tb{4} & \tb{4}32 \\
            \tb{3}2 & \tb{1}
        \end{ytableau} &
        \begin{ytableau}
            \tb{4} & \tb{4}32 \\
            \tb{3}21 & \tb{1}
        \end{ytableau} \\
        \ol{\mf{P}}_{(0,0,2,2)} & \beta\cdot\ol{\mf{P}}_{(0,1,2,2)} & \beta^2\cdot\ol{\mf{P}}_{(0,2,2,2)} & \beta^3\cdot\ol{\mf{P}}_{(1,2,2,2)} & \beta^4\cdot\ol{\mf{P}}_{(2,2,2,2)} \\ \\
        \begin{ytableau}
            \tb{4} & \tb{2} \\
            \tb{3} & \tb{3}
        \end{ytableau} & \begin{ytableau}
            \tb{4} & \tb{4}3 \\
            \tb{3} & \tb{1}
        \end{ytableau} &
        \begin{ytableau}
            \tb{4} & \tb{4}3 \\
            \tb{3}1 & \tb{1}
        \end{ytableau} & & \\
        \ol{\mf{P}}_{(0,1,2,1)} & \beta \cdot\ol{\mf{P}}_{(1,0,2,2)} & \beta^2\cdot\ol{\mf{P}}_{(2,0,2,2)} & & \\ \\
        \begin{ytableau}
            \tb{4} & \tb{1} \\
            \tb{3} & \tb{3}
        \end{ytableau} &
        \begin{ytableau}
            \tb{4}2 & \tb{2} \\
            \tb{3} & \tb{3}
        \end{ytableau} & & \\
        \ol{\mf{P}}_{(1,0,2,1)} & \beta\cdot\ol{\mf{P}}_{(0,2,2,1)} & & \\ \\
        & \begin{ytableau}
            \tb{4}1 & \tb{1} \\
            \tb{3} & \tb{3}
        \end{ytableau} & & \\
        & \beta\cdot\ol{\mf{P}}_{(2,0,2,1)} & &
    \end{array}$$
\end{example}


\subsection{Glide polynomials \texorpdfstring{$\ol{\mf{F}}_a$}{F} and quasiLascoux polynomials \texorpdfstring{$\ol{\mf{Q}}_a$}{Q}}

We will now define the last two types of polynomials and explain the transition formula between them. For weak compositions $a$ and $b$, $a^+$ denotes the nonzero elements of $a.$ We say that $b$ \emph{\tb{\tcb{dominates}}} $a$ and write $b \ge a$ if $b_1 + \dots + b_i \ge a_1 + \dots + a_i$ for every $i.$ Then the \emph{\tb{\tcb{glide polynomial}}} $\ol{\mf{F}}_a$ is $$\ol{\mf{F}}_a := \sum_{\substack{b\ge a \\ b^+ = a^+}}\ol{\mf{P}}_a,$$ and analogously, the \emph{\tb{\tcb{quasiLascoux polynomial}}} $\ol{\mf{Q}}_a$ is $$\ol{\mf{Q}}_a := \sum_{\substack{b\ge a \\ b^+=a^+}} \ol{\mf{A}}_a.$$ A \emph{\tb{\tcb{set-valued quasi-skyline filling}}} of shape $a$ is the same as a semistandard set-valued filling, except that (S.5) is replaced with the following property:
\begin{itemize}[leftmargin=1.5cm]
    \item[(S.$5'$)] The leftmost anchor in row $i$ is at most $i,$ and the anchors in the left column decrease from top to bottom.
\end{itemize}
We write $\ol{\mf{Q}}{\sf{SSF}}(a)$ for the set of set-valued quasi-skyline fillings of shape $a.$ Then $\ol{\mf{Q}}{\sf{SSF}}(a)$ is a superset of $\ol{\mf{A}}{\sf{SSF}}(a)$ of $a$, and in fact $$\ol{\mf{Q}}{\sf{SSF}}(a) = \bigcup_{\substack{b\ge a \\ b^+ = a^+}} \ol{\mf{A}}{\sf{SSF}}(b).$$ It follows that we can alternately define $\ol{\mf{Q}}_a$ as $$\ol{\mf{Q}}_a = \sum_{T\in \ol{\mf{Q}}{\sf{SSF}}(a)} \beta^{|T|-|a|}\mathbf{x}^{\tn{wt}(T)}.$$ A tableau $T\in \ol{\mf{Q}}{\sf{SSF}}(a)$ is \emph{\tb{\tcb{quasiYamanouchi}}} if for every $i$ in $T$, at least one of the following holds:
\begin{itemize}
    \item The leftmost $i$ is the anchor in the leftmost column of row $i.$
    \item There is an $i+1$ weakly to the right of the leftmost $i$ and in a different box.
\end{itemize}
We write $\ol{\mf{Q}}2\ol{\mf{F}}(a)$ for the set of quasiYamanouchi tableau of shape $a.$ We can now state the second expansion formula that we will need:

\begin{theorem}[Monical, Pechenik, and Searles \cite{Monical_Pechenik_Searles_2021}, Theorem 4.12]\label{thm:q2f}
    The glide expansion for the quasiLascoux polynomials is $$\ol{\mf{Q}}_a = \sum_{T\in \ol{\mf{Q}}2\ol{\mf{F}}(a)} \beta^{|T|-|a|}\ol{\mf{F}}_{\tn{wt}(T)}.$$
\end{theorem}

\begin{example}\label{ex:q2f_0}
    Returning to $a=(0,0,1,2)$ as in Example \ref{ex:a2p_0}, the two tableaux from the second row in Example \ref{ex:a2p_0} are in $\ol{\mf{A}}2\ol{\mf{P}}(0,0,1,2)$ but not in $\ol{\mf{Q}}2\ol{\mf{F}}(0,0,1,2)$ because they have a 1 that is not an anchor in the left column of its own row but do not have any 2's, violating the quasiYamanouchi property. Thus, the tableaux in the glide expansion of $\ol{\mf{Q}}_{(0,0,1,2)}$ are the ones shown below, with the corresponding terms of the expansion listed below them. Again, we omit the bottom two empty rows of each tableau.
        $$\arraycolsep=12pt\def\arraystretch{1.5}\begin{array}{llllll}
        \ytableausetup{boxsize=2em}
        \begin{ytableau}
            \tb{4} & \tb{4} \\
            \tb{3}
        \end{ytableau} & 
        \begin{ytableau}
            \tb{4} & \tb{4}3 \\
            \tb{3}
        \end{ytableau} &
        \begin{ytableau}
            \tb{4} & \tb{4}3 \\
            \tb{3}2
        \end{ytableau} &
        \begin{ytableau}
            \tb{4} & \tb{4}32 \\
            \tb{3}2
        \end{ytableau} &
        \begin{ytableau}
            \tb{4} & \tb{4}32 \\
            \tb{3}21
        \end{ytableau} &
        \begin{ytableau}
            \tb{4} & \tb{4}321 \\
            \tb{3}21
        \end{ytableau}\\
        \ol{\mf{F}}_{(0,0,1,2)} & \beta\cdot \ol{\mf{F}}_{(0,0,2,2)} & \beta^2\cdot\ol{\mf{F}}_{(0,1,2,2)} & \beta^3\cdot\ol{\mf{F}}_{(0,2,2,2)} & \beta^4\cdot \ol{\mf{F}}_{(1,2,2,2)} & \beta^5\cdot\ol{\mf{F}}_{(2,2,2,2)}
        \end{array}$$
\end{example}

\begin{example}\label{ex:q2f_1}
    Using $a=(0,0,2,2)$ as in Example \ref{ex:a2p_1}, the four tableaux with a 1 but not a 2 are no longer valid, as they violate the quasiYamanouchi property. Thus, the tableaux in $\ol{\mf{Q}}2\ol{\mf{F}}(0,0,2,2)$ and their corresponding monomials in the glide expansion of $\ol{\mf{Q}}_{(0,0,2,2)}$ are the ones shown below. As in our other examples, the two empty rows at the bottom are omitted.
    $$\arraycolsep=12pt\def\arraystretch{1.5}\begin{array}{llllll}
        \ytableausetup{boxsize=2em}
        \begin{ytableau}
            \tb{4} & \tb{4} \\
            \tb{3} & \tb{3}
        \end{ytableau} & 
        \begin{ytableau}
            \tb{4} & \tb{4}3 \\
            \tb{3} & \tb{2}
        \end{ytableau} &
        \begin{ytableau}
            \tb{4} & \tb{4}3 \\
            \tb{3}2 & \tb{2}
        \end{ytableau} &
        \begin{ytableau}
            \tb{4} & \tb{4}32 \\
            \tb{3}2 & \tb{1}
        \end{ytableau} &
        \begin{ytableau}
            \tb{4} & \tb{4}32 \\
            \tb{3}21 & \tb{1}
        \end{ytableau} \\
        \ol{\mf{F}}_{(0,0,2,2)} & \beta\cdot\ol{\mf{F}}_{(0,1,2,2)} & \beta^2\cdot\ol{\mf{F}}_{(0,2,2,2)} & \beta^3\cdot\ol{\mf{F}}_{(1,2,2,2)} & \beta^4\cdot\ol{\mf{F}}_{(2,2,2,2)} \\ \\
        \begin{ytableau}
            \tb{4} & \tb{2} \\
            \tb{3} & \tb{3}
        \end{ytableau} & 
        \begin{ytableau}
            \tb{4}2 & \tb{2} \\
            \tb{3} & \tb{3}
        \end{ytableau} & & \\
        \ol{\mf{F}}_{(0,1,2,1)} & \beta\cdot\ol{\mf{F}}_{(0,2,2,1)} & &
        \end{array}$$
\end{example}

\section{Proof of Theorem \ref{thm:alt_sum} and examples of the involution}\label{sec:alt_sum_proof}

\subsection{Proof for \texorpdfstring{$\sum_b Q_b^a(-1)$}{Q sum}}

    By Theorem \ref{thm:a2p}, $$\sum_b Q_b^a(-1) = \sum_{T\in \ol{\mf{A}}2\ol{\mf{P}}(a)} (-1)^{|T|-|a|}.$$ To prove that this equals 0 or 1, construct the following map $\iota$ on the set $T\in \ol{\mf{A}}2\ol{\mf{P}}(a)$:
    \begin{enumerate}
        \item Set $m$ to be the smallest entry appearing in $T$.
        \item Find the rightmost column of $T$ where a free $m$ can be either added or removed to get a new tableau that is still in $\ol{\mf{A}}2\ol{\mf{P}}(a)$, and add or remove $m$ from that column. 
        \item If no such column exists, reset $m$ to be the next smallest entry $m^\uparrow$ and repeat Step 2. Continue until either Step 2 succeeds or $m$ is equal to the maximum entry of $T$, in which case we set $\iota(T):=T$.
    \end{enumerate}
    We can see that $\iota$ is well-defined on $\ol{\mf{A}}2\ol{\mf{P}}(a)$ because removing an $m$ in a given column can only be done in one way, and (S.4) guarantees that adding a free $m$ can also be done in only one way. Also, whenever $\iota(T)\ne T$, it is sign reversing since $|\iota(T)| = |T| \pm 1,$ as we are either adding or removing one entry of $T$ to get $\iota(T)$. We will show that $\iota$ is an involution and that there is at most one tableau $T$ for which $\iota(T) = T$, which will imply Theorem \ref{thm:a2p}. To show these facts, we use the following lemma:

    \begin{lemma}\label{lem:remove_min}
        If $\iota(T)\ne T$, the entry $m$ that is added or removed to get from $T$ to $\iota(T)$ is the largest entry such that for every $m' < m$ in $T$, row $m'$ has all entries equal to $m'$ and is weakly longer than all rows above it. Moreover, if the maximal $m$ satisfying that condition is not the largest element of $T$, then Step 2 succeeds and hence $\iota(T)\ne T$.
    \end{lemma}

    \begin{proof}
        First we note that for the choice of $m$ described in Lemma \ref{lem:remove_min}, no $m' < m$ can possibly be added to or removed from a column of $T$. We cannot remove an $m'$ because every $m'$ is an anchor. We also cannot add a free $m'$ because the only columns that do not already contain an $m'$ are those in rows longer than row $m'$, meaning rows below row $m'$, and those rows have anchors less than $m'$, meaning adding a free $m'$ would violate (S.2). This shows that the $m$ chosen for $\iota$ must be at least the $m$ described in Lemma \ref{lem:remove_min}.

        It remains to show that this choice of $m$ actually can be added to or removed from some column of $T$. We can remove all rows of $T$ below row $m$ without changing which entry of $T$ will be added or removed, since no entry from one of those rows will be the entry chosen or impact which entry is chosen. Thus, we may assume that $m$ is the smallest entry of $T$. We consider several cases: 
        \begin{itemize}
            \item If $T$ contains multiple $m$'s and at least one of them is free, then we can remove the rightmost free $m$ without violating any conditions. Thus, we can assume that either $T$ contains only one $m,$ or all $m$'s in $T$ are anchors.
            \item Suppose $T$ contains only one $m$, the $m$ is free, and it is not in the rightmost column. Then we can add a free $m$ to the rightmost column without violating any conditions.
            \item Suppose $T$ contains only one $m$, the $m$ is free, and it is in the rightmost column. Then by the meson-highest condition, $m^\uparrow$ must also appear in that same rightmost column in a different box from $m.$ However, if $m^\uparrow$ is an anchor in that column, then $m$ should instead be in the same box of the column as $m^\uparrow$ by (S.4), since $m^\uparrow$ is the smallest anchor in that column. Similarly, if $m^\uparrow$ is not an anchor, then both $m$ and $m^\uparrow$ must be in whichever box of the last column has the smallest anchor, again by (S.4). This contradicts $m$ and $m^\uparrow$ being in different boxes, so this case is impossible.
            \item Suppose now that every $m$ in $T$ is an anchor, but that not every column of $T$ contains an $m.$ Consider the rightmost column not containing an $m,$ say column $c.$ If $c$ is not the rightmost column of $T$, then column $c+1$ contains an $m$ as an anchor. This means we can add a free $m$ to column $c,$ since adding an $m$ to the box immediately to the left of the one containing the $m$ in column $c+1$ will not violate (S.2) or any other conditions. If $c$ is the rightmost column, then again a free $m$ can be added to column $c$ without violating any conditions.
            \item The only remaining case is that every column of $T$ contains an $m$ and all the $m$'s are anchors. The $m$ in the leftmost column must then be in row $m,$ so row $m$ must be nonempty, and it must be the lowest row of $T$ since $m$ is the smallest entry. By (S.2) and minimality of $m,$ every other box of row $m$ must then contain an $m$ and no free entries. The only thing left to check is that no $m$ appears in a higher row, or equivalently, row $m$ is at least as long as every other row. If not, let $c$ be the rightmost column of row $m.$ Then column $c+1$ contains an $m$ in some higher row. Taking that $m$ together with the box to its left and the $m$ in column $c$ then creates an upside-down L shape violating (S.3), since both $m$'s are anchors and thus $\beta = \gamma = m.$ This is a contradiction, so we conclude that row $m$ has maximal length. But then $m$ is not the maximal entry satisfying the condition of Lemma \ref{lem:remove_min}, since $m^\uparrow$ also satisfies the condition. Thus, this case is not possible.
        \end{itemize}    
        Thus, in each possible case, some $m$ can be added to or removed from $T$, which shows that the $m$ chosen is indeed the one described in the lemma statement.
    \end{proof}
    
    \begin{lemma}
        The map $\iota$ is an involution on $\ol{\mf{A}}2\ol{\mf{P}}(a)$.
    \end{lemma}

    \begin{proof}
    Note that as long as the same $m$ and the same column is chosen when applying $\iota$ to $\iota(T)$ as when applying it to $T$, we will get $\iota^2(T) = T,$ since the operations of adding and removing an $m$ from the same column cancel out. Thus, we consider the possible ways that a different $m$ or column could be chosen when computing $\iota(\iota(T))$:
    \begin{itemize}
        \item The same $m$ is chosen, but in a different column. Whichever column was used when finding $\iota(T)$ is certainly also a valid column to modify to find $\iota(\iota(T))$, so the only potential issue is if there is a column of $\iota(T)$ further to the right in which we can add or remove an $m.$ But then that same column would also have been modified when finding $\iota(T)$ instead of the one that was modified, because the existence or nonexistence of an extra $m$ in some column further to the left will not impact whether or not (S.1)-(S.5) and the meson-highest condition hold if an $m$ is added or removed from a column further to the right, so modifying the further right column is valid for $\iota(T)$ if and only if it is valid for $T$. This is a contradiction, so provided the same $m$ is chosen, the same column will also be chosen.
        
        \item A different $m$ is chosen. One potential issue is that after adding or removing an $m,$ a smaller value $m'$ can now be either added to or removed from $T$, but this cannot happen because $m'$ is the value chosen if and only if the condition in Lemma \ref{lem:remove_min} applies to $m'$, and that condition is not impacted by adding or removing an entry of $T$ that is larger than $m'.$ The only other potential issue is that there is only one $m$ in $T$ and it gets removed when going from $T$ to $\iota(T)$, so $\iota(T)$ no longer contains an $m.$ For this to happen, the lone $m$ must be in the rightmost column of $T$, because otherwise it can be added to the rightmost column without violating any of the conditions. It must also be free, since otherwise it could not be removed. But by the argument in the proof of Lemma \ref{lem:remove_min}, it is not possible that there is only one $m$ in $T$ and that the lone $m$ is a free entry in the far right column. Thus, this case cannot actually happen.
    \end{itemize}
    Thus, $\iota$ is an involution. 
    \end{proof}
    
    To complete the proof for the $\sum_b Q_b^a(-1)$ case, it remains to show:

    \begin{lemma}
        We have $\iota(T)=T$ for exactly one $T\in \ap(a)$ if the nonzero parts of $a$ are nonincreasing, and $\iota(T)\ne T$ for all $T \in \ap(a)$ otherwise.
    \end{lemma}

    \begin{proof}
        Note that the condition in Lemma \ref{lem:remove_min} is vacuously satisfied by the smallest entry $m$ appearing in $T$, since then there are no entries $m'<m$ to consider. Thus, the only way we fail to find an $m$ in Step 2 is if the maximal $m$ we get from Lemma \ref{lem:remove_min} is the largest entry of $T$, which means the condition in Lemma \ref{lem:remove_min} applies to \emph{every} entry $m$ of $T$. This is equivalent to saying that all nonempty rows of $T$ are weakly decreasing in length and that every box of $T$ is filled with a single entry equal to the corresponding row number. No such tableau $T$ exists if the nonzero parts of $a$ are not weakly decreasing, and exactly one such $T$ exists if they are weakly decreasing.
    \end{proof}

    Thus, the sum is 0 if the parts of $a$ are not weakly decreasing, and if they are, the sum is 1 because the single unpaired $T$ has no free entries and thus has positive sign since $|T|-|a|=0.$ \qed

\subsection{Proof for \texorpdfstring{$\sum_b M_b^a(-1)$}{M sum}}

    The proof for $\sum_b M_a^b(-1)$ is essentially identical. Instead of Theorem \ref{thm:a2p}, we use Theorem \ref{thm:q2f} to get $$\sum_b M_b^a(-1) = \sum_{T\in \ol{\mf{Q}}2\ol{\mf{F}}(a)}\beta^{|T|-|a|}\ol{\mf{F}}_{\tn{wt}(T)}.$$ We then define $\iota$ in the same manner as before, except that it is an involution on $\ol{\mf{Q}}2\ol{\mf{F}}$. The only difference in the proof is that if $m$ is the entry chosen by $\iota$ and $m' < m$ is another entry of $T$, it could be the case that instead of all the occurrences of $m'$ appearing as anchors in row $m',$ they all appear as anchors in some higher row. Thus, $m$ will instead be the smallest entry such that for every $m' < m$, all occurrences of $m'$ appear as anchors in the same row, and that row is weakly longer than all rows above it. 
    
    So, we will have $\iota(T)=T$ if and only if the nonzero elements of $a$ are weakly decreasing, $T$ has no free entries, and every row of $T$ contains only repeats of the same entry. We claim that in fact the repeated entry in each row must equal the row index. If not, there is some $m$ that is the anchor of a row higher than row $m.$ Then by the quasiYamanouchi condition, $m^\uparrow = m+1,$ since $m$ is not an anchor in its own row. Then $m+1$ must also be the anchor of some row higher than its own row. By the same logic, we must have $(m+1)^\uparrow = m+2,$ and $m+2$ must then be the anchor of a row higher than row $m+2.$ Continuing this logic inductively, we find that the largest entry $M$ appearing in $T$ is also an anchor of some row higher than row $M$. But then the quasiYamanouchi condition requires $T$ to contain an entry equal to $M+1$, which is a contradiction. So in fact, every entry $m$ of a tableau $T$ for which $\iota(T)=T$ must be an anchor in its own row $m$ rather than in another row. 
    
    This shows that again, there is exactly one unpaired tableau $T$ in the case where the parts of $a$ are weakly decreasing and no unpaired tableaux otherwise. As with $\sum_b Q_b^a(-1)$, in the case where such an unpaired tableau exists, it appears with positive sign since it has no free entries. Thus, we get that $\sum_b M_b^a(-1) = 1$ if the parts of $a$ are weakly decreasing and $\sum_b M_b^a(-1) = 0$ otherwise, as desired. This completes the proof of Theorem \ref{thm:alt_sum}. \qed

\bigskip

We end with several examples to illustrate the involution.



\subsection{Example where \texorpdfstring{$\sum_b Q_b^a(-1) = \sum_b M_b^a(-1) = 0$}{Q=M=0}}

\begin{example}\label{ex:a2p_involution}
    Revisiting $a=(0,0,1,2)$ as in Examples \ref{ex:a2p_0} and \ref{ex:q2f_0}, we get
    \begin{align*}
        \sum_b Q_b^a(\beta) &= 1+\beta + 2\beta^2 + 2\beta^3 + \beta^4 + \beta^5, \\
        \sum_b M_b^a(\beta) &= 1+\beta+\beta^2+\beta^3+\beta^4+\beta^5.
    \end{align*}
    Plugging in $\beta=-1$ gives
    \begin{align*}
        \sum_b Q_b^a(\beta) &= 1-1+2-2+1-1 = 0, \\
        \sum_b M_b^a(\beta) &= 1-1+1-1+1-1 = 0.
    \end{align*}
    This matches Theorem \ref{thm:alt_sum} since the nonzero parts of $a=(0,0,1,2)$ are not weakly decreasing, so we expect the sums to be 0. The pairing induced by $\iota$ for the tableaux in $\ol{\mf{A}}_2\ol{\mf{P}}(0,0,1,2)$ is shown below:
    $$\arraycolsep=6pt\begin{array}{llllllllllll}
        \ytableausetup{boxsize=2em}
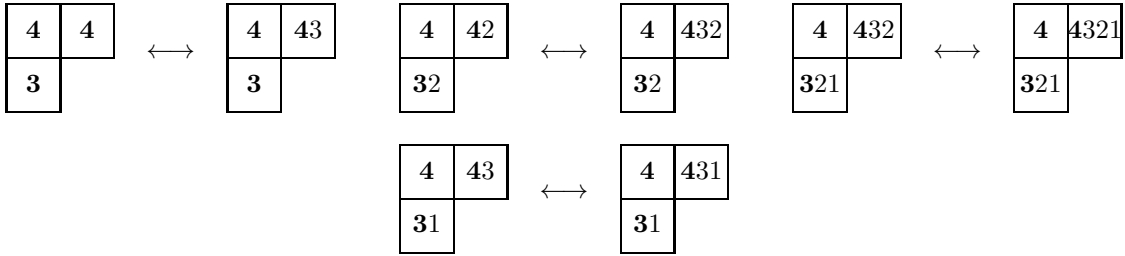
                \begin{ytableau}
            \tb{4} & \tb{4} \\
            \tb{3}
        \end{ytableau} & \longleftrightarrow &
        \begin{ytableau}
            \tb{4} & \tb{4}3 \\
            \tb{3}
        \end{ytableau} & &
        \begin{ytableau}
            \tb{4} & \tb{4}3 \\
            \tb{3}2
        \end{ytableau} & \longleftrightarrow &
        \begin{ytableau}
            \tb{4} & \tb{4}32 \\
            \tb{3}2
        \end{ytableau} & &
        \begin{ytableau}
            \tb{4} & \tb{4}32 \\
            \tb{3}21
        \end{ytableau} & \longleftrightarrow &
        \begin{ytableau}
            \tb{4} & \tb{4}321 \\
            \tb{3}21
        \end{ytableau}\\ \\
        & & & & \begin{ytableau}
            \tb{4} & \tb{4}3 \\
            \tb{3}1
        \end{ytableau} & \longleftrightarrow &
        \begin{ytableau}
            \tb{4} & \tb{4}31 \\
            \tb{3}1
        \end{ytableau} \\
    \end{array}$$
    The pairing of the tableaux in $\ol{\mf{Q}}2\ol{\mf{F}}(0,0,1,2)$ is the same except with the two tableaux in the bottom row omitted.
\end{example}



\subsection{Example where \texorpdfstring{$\sum_b M_b^a(-1)=\sum_b Q_b^a(-1) = 1$}{Q=M=1}}

\begin{example}\label{ex:q2f_involution}
    Revisiting $a=(0,0,2,2)$ as in Examples \ref{ex:a2p_1} and \ref{ex:q2f_1}, we get 
    \begin{align*}
        \sum_b Q_b^a(\beta) &= 3+4\beta+2\beta^2+\beta^3 + \beta^4, \\
        \sum_b M_b^a(\beta) &= 2+2\beta+\beta^2+\beta^3+\beta^4.
    \end{align*}
    When $\beta=-1,$ these sums become
    \begin{align*}
        \sum_b Q_b^a(-1) &= 3-4+2-1+1 = 1, \\
        \sum_b M_b^a(-1) &= 2-2+1-1+1 = 1.
    \end{align*}
    This again matches Theorem \ref{thm:alt_sum} because the nonzero parts of $a=(0,0,2,2)$ are in weakly decreasing order, so we expect both sums to be 1. The pairing of the tableaux in $\ol{\mf{A}}2\ol{\mf{P}}(0,0,2,2)$ induced by $\iota$ is shown below. Note that the top left tableau is the unique one that is unpaired.
    $$\arraycolsep=6pt\begin{array}{lllllllll}
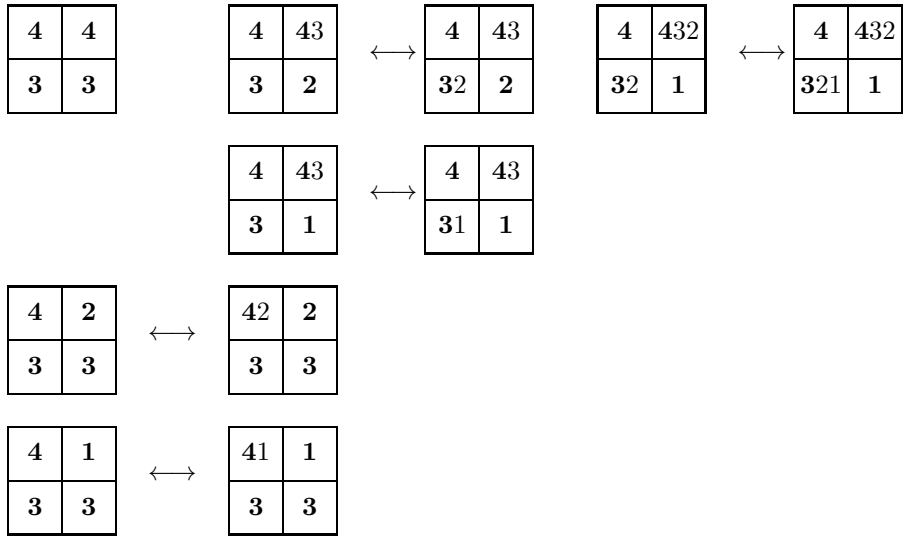
        \begin{ytableau}
            \tb{4} & \tb{4} \\
            \tb{3} & \tb{3}
        \end{ytableau} & &
        \begin{ytableau}
            \tb{4} & \tb{4}3 \\
            \tb{3} & \tb{2}
        \end{ytableau} & \longleftrightarrow
        \begin{ytableau}
            \tb{4} & \tb{4}3 \\
            \tb{3}2 & \tb{2}
        \end{ytableau} & &
        \begin{ytableau}
            \tb{4} & \tb{4}32 \\
            \tb{3}2 & \tb{1}
        \end{ytableau} & \longleftrightarrow
        \begin{ytableau}
            \tb{4} & \tb{4}32 \\
            \tb{3}21 & \tb{1}
        \end{ytableau} \\ \\
        & & \begin{ytableau}
            \tb{4} & \tb{4}3 \\
            \tb{3} & \tb{1}
        \end{ytableau} & \longleftrightarrow
        \begin{ytableau}
            \tb{4} & \tb{4}3 \\
            \tb{3}1 & \tb{1}
        \end{ytableau} & & \\ \\
        \begin{ytableau}
            \tb{4} & \tb{2} \\
            \tb{3} & \tb{3}
        \end{ytableau} & \longleftrightarrow &
        \begin{ytableau}
            \tb{4}2 & \tb{2} \\
            \tb{3} & \tb{3}
        \end{ytableau} \\ \\
        \begin{ytableau}
            \tb{4} & \tb{1} \\
            \tb{3} & \tb{3}
        \end{ytableau} & \longleftrightarrow &
        \begin{ytableau}
            \tb{4}1 & \tb{1} \\ 
            \tb{3} & \tb{3}
        \end{ytableau}
    \end{array}$$
    The pairing of tableaux in $\ol{\mf{Q}}2\ol{\mf{F}}(0,0,2,2)$ is the same except without the four tableaux containing 1's but not 2's (i.e. the tableaux in the second and fourth rows from the top).
\end{example}

\subsection{Example of a tableau needing multiple iterations of Step 2}

\begin{example}
    Consider the following tableau $T\in \ol{\mf{Q}}2\ol{\mf{F}}(3,1,2,4,5) \se \ol{\mf{A}}2\ol{\mf{P}}(3,1,2,4,5)$:
    $$\begin{ytableau}
        \tb{5} & \tb{5}4 & \tb{4} \\
        \tb{4} \\
        \tb{3} & \tb{3} \\
        \tb{2} & \tb{2} & \tb{2} & \tb{2} \\
        \tb{1} & \tb{1} & \tb{1} & \tb{1} & \tb{1}
    \end{ytableau}$$
    No free 1 can be added to or removed from $T$, since row 1 consists only of 1's and is weakly longer than all rows above it (as in Lemma \ref{lem:remove_min}), and no free 2 can be added to or removed from $T$ for the same reason. However, a free 3 can be added to column 3, since row 1 is longer than row 3. Thus, $m=3$ is the value chosen in Step 2, and the tableau $\iota(T)$ is as shown below:
    $$\begin{ytableau}
        \tb{5} & \tb{5}4 & \tb{4}3 \\
        \tb{4} \\
        \tb{3} & \tb{3} \\
        \tb{2} & \tb{2} & \tb{2} & \tb{2} \\
        \tb{1} & \tb{1} & \tb{1} & \tb{1} & \tb{1}
    \end{ytableau}$$
\end{example}

\section*{Acknowledgments}

I'm thankful to Oliver Pechenik for suggesting the problem, and to Oliver Pechenik, Karen Yeats, Sophie Spirkl, and the anonymous reviewers for providing helpful comments. I was partially supported by the Natural Sciences and Engineering Research Council of Canada (NSERC) grant RGPIN-2022-03093.

\printbibliography

\end{document}